\documentclass[11pt]{amsart}
\def\cal#1{\mathcal{#1}}

          \usepackage{amssymb}
          \usepackage{amsmath}
          \usepackage{amsthm}
          \usepackage{amsfonts}
          \usepackage[english]{babel}
          \usepackage[utf8]{inputenc}
          \usepackage{enumerate}
          \usepackage{graphicx}
          \usepackage{url}
          \usepackage{color}

\usepackage[numbers,sort&compress]{natbib}
\def\cal{\mathcal}
\newcommand{\comment}[1]{}
\newcommand{\ind}{{\bf 1}}
\def\indd#1{{\ind}_{\{#1\}}}

\def\indn#1{\{#1_n\}_{n\in\N}}
\newcommand{\proba}{\mathbb P}
\newcommand{\esp}{{\mathbb E}}

\newcommand{\defe}{\mathrel{\mathop:}=}

\newcommand{\cov}{{\rm{Cov}}}
\newcommand{\var}{{\rm{Var}}}

\newcommand{\calA}{{\cal A}}
\newcommand{\calB}{{\cal B}}

\newcommand{\filF}{{\cal F}}

\newcommand{\calG}{{\cal G}}

\newcommand{\calN}{{\cal N}}

\def\B{{\mathbb B}}

\def\G{{\mathbb G}}

\def\T{{\mathbb T}}

\newcommand{\eqnh}{\begin{eqnarray*}}
\newcommand{\eqne}{\end{eqnarray*}}
\newcommand{\eqnhn}{\begin{eqnarray}}
\newcommand{\eqnen}{\end{eqnarray}}
\newcommand{\equh}{\begin{equation}}
\newcommand{\eque}{\end{equation}}

\def\summ#1#2#3{\sum_{#1 = #2}^{#3}}
\def\prodd#1#2#3{\prod_{#1 = #2}^{#3}}
\def\sif#1#2{\sum_{#1=#2}^\infty}

\def\topp#1{^{(#1)}}

\def\nn#1{{\left\|#1\right\|}}
\def\snn#1{\|#1\|}

\def\abs#1{\left|#1\right|}
\def\sabs#1{|#1|}

\def\ccbb#1{\left\{#1\right\}}

\def\spp#1{(#1)}
\def\pp#1{\left(#1\right)} 
 
\def\bb#1{\left[#1\right]}

\def\mmid{\;\middle\vert\;}
\def\ip#1{\left\langle#1\right\rangle}

\def\floor#1{\left\lfloor #1 \right\rfloor}

\def\vv#1{{\boldsymbol #1}}

\def\mand{\mbox{ and }}
\def\qmand{\quad\mbox{ and }\quad}

\def\qmwith{\quad\mbox{ with }\quad}
\def\mfa{\mbox{ for all }}

\def\mmas{\mbox{ as }}

\def\wt#1{\widetilde{#1}}

\def\weakto{\Rightarrow}

\def\limn{\lim_{n\to\infty}}
\def\limm{\lim_{m\to\infty}}

\def\limsupn{\limsup_{n\to\infty}}
\def\liminfn{\liminf_{n\to\infty}}

\def\Z{{\mathbb Z}}
\def\Zd{{\mathbb Z^d}}
\def\R{{\mathbb R}}
\def\Rd{{\mathbb R^d}}

\def\N{{\mathbb N}}

\usepackage{epsfig}
\usepackage{ifthen}

\newtheorem{Thm}{Theorem}[section]
\newtheorem{Lem}[Thm]{Lemma}
\newtheorem{Prop}[Thm]{Proposition}
\newtheorem{Coro}[Thm]{Corollary}
\newtheorem{Assumption}[Thm]{Assumption}

\theoremstyle{definition}
\newtheorem{Rem}[Thm]{Remark}

\newtheorem{Example}[Thm]{Example}

\numberwithin{equation}{section}

\title[Weighted Bernoulli random fields]{Limit theorems for weighted Bernoulli random fields under Hannan's condition}

\author[Klicnarov\'a]{Jana Klicnarov{\'a}}
\address{
Jana Klicnarov\'a,
Faculty of Economics, University of South Bohemia, Studentsk\'a 13, 370 05 \v Cesk\'e Bud\v ejovice, 
Czech Republic.
}
\email{opravdova.kacka@email.cz}

\author[Voln\'y]{Dalibor Voln{\'y}}
\address
{
Dalibor Voln\'y,
Laboratoire de Math\'ematiques Rapha\"el Salem,
Universit\'e de Rouen, 
76801, Saint Etienne du Rouvray, France.
}
\email{dalibor.volny@univ-rouen.fr}

\author[Wang]{Yizao Wang}
\address
{
Yizao Wang,
Department of Mathematical Sciences,
University of Cincinnati,
2815 Commons Way,
Cincinnati, OH, 45221-0025.
}
\email{yizao.wang@uc.edu}

\begin{document}\sloppy
\begin{abstract}Recently, invariance principles for partial sums of Bernoulli random fields over rectangular index sets have been proved under Hannan's condition. In this note we complement previous results by establishing limit theorems for weighted Bernoulli random fields, including central limit theorems for partial sums over arbitrary index sets and invariance principles for Gaussian random fields. Most results improve earlier ones on Bernoulli random fields under  Wu's condition, which is stronger than Hannan's condition.
\end{abstract}
\maketitle

\section{Introduction}

We are interested in limit theorems for partial sums of weighted stationary random fields $\{X_j\}_{j\in\Zd}$, in form of
\equh\label{eq:1}
S_n = \sum_{j\in\Zd}b_{n,j}X_j, n\in\N,
\eque
where $\{b_{n,j}\}_{j\in\Zd}$ are coefficients such that $\sum_jb_{n,j}^2<\infty$. We will impose further conditions on the dependence of $\{X_j\}_{j\in\Zd}$ so that $S_n$ is well defined in the $L^2$ sense.

Limit theorems for partial sums of dependent random variables have a long history. In the case $d=1$, limit theorems for stationary sequences have been extensively developed. The most considered case is the unweighted case with $b_{n,j} = \indd{j\in\{1,\dots,n\}}$, which yields
\[
S_n = X_1+\cdots+X_n.
\] In this case, various conditions on the weak dependence of stationary sequence $\{X_j\}_{j\in\Z}$ have been developed. Heuristically, when  the sequence of $\{X_j\}_{j\in\Z}$ is said to have weak dependence, $S_n$ should behave asymptotically as partial sums of i.i.d.~random variables; in particular one expects
\equh\label{eq:WIPd=1}
\ccbb{\frac{S_{\floor {nt}}}{\sqrt n}}_{t\in[0,1]}\weakto\sigma\{\B_t\}_{t\in[0,1]}
\eque
for a standard Brownian motion $\B$ and some constant $\sigma>0$. Results of this type are referred to as (weak) invariance principles or functional central limit theorems. See for examples \citep{dedecker00functional,maxwell00central,merlevede06recent,wu05nonlinear,peligrad97central,peligrad06central} and references therein for recent developments. In case that $b_{n,j}$ is a more general $\ell^2$ sequence, the same rational applies: $S_n$ in~\eqref{eq:1} with weakly dependent $\{X_j\}_{j\in\Z}$ should behave as if $\{X_j\}_{j\in\Z}$ are independent and identically distributed (i.i.d.) random variables. Different types of invariance principles may arise, see for example \citep{dedecker11invariance} where the limiting processes are fractional Brownian motions.

A few conditions have been known that lead to the above invariance principle~\eqref{eq:WIPd=1}. See for example \citep{durieu08comparison,durieu09independence} on comparisons between different conditions. In this paper we focus on Hannan's condition~\citep{hannan73central}  and Wu's condition~\citep{wu05nonlinear}, and particular their extensions to random fields. In order to compare with our results for random fields, assume in addition that the stationary random variables $\{X_j\}_{j\in\Z}$ have the form
\[
X_j = f\circ T^j(\{\epsilon_k\}_{k\in\Z})
\]
for some measurable function
 $f:\R^\Z\to\R$ and the shift operator $T$ on $\R^\Z$ defined by $[T_j(w)]_k = w_{j+k}$ for $w = \{w_k\}_{k\in\Z}\in\R^\Z$, and a sequence of stationary random variables $\{\epsilon_k\}_{k\in\Z}$. We assume $\esp X_j = 0$ and $\esp X_j^2<\infty$. Furthermore, introduce $\filF_j := \sigma(\epsilon_k:k\leq j)$, $\filF_{-\infty} := \bigcap_j\filF_j$ and $\filF_\infty := \bigvee_j\filF_j$, and assume $f$ to be regular in the sense that $\esp (f\mid\filF_{-\infty}) = 0, \esp (f\mid\filF_\infty) = f$. 
 With these notations, Hannan's condition reads as
 \equh\label{eq:Hannand=1}
 \sum_{j\in\Z}\snn{\esp(f\circ T^j\mid\filF_0)- \esp(f\circ T^j\mid \filF_{-1})}_2<\infty.
 \eque
To introduce Wu's condition, furthermore set first $\epsilon^*_j = \epsilon_j$ for $j\neq 0$ and $\epsilon^*_0$ to be an independent copy of $\epsilon_0$ and independent of all the other random variables, and then set $X_j^* = f\circ T^j(\{\epsilon^*_k\}_{k\in\Z})$. Then Wu's condition reads as
\equh\label{eq:Wud=1}
\sum_{j\in\Z}\snn{X_j - X^*_j}_2<\infty.
\eque

Both conditions have been extensively investigated (e.g.~\citep{cuny13quenched,dedecker07weak,wu11asymptotic}). 
It has been shown in \citep{volny14invariance} that 
Wu's condition is strictly stronger than Hannan's condition, in the sense that 
one can find an example such that the invariance principle as in~\eqref{eq:WIPd=1} holds for Hannan's condition, but Wu's condition is violated.  On the other hand, Wu's condition is very practical in proving limit theorems~\citep{wu11asymptotic}, and conditions of similar types have lead to strong invariance principles \citep{wu07strong,berkes14komlos}.

Limit theorems for stationary random fields ($d\ge2$) have also been investigated since long time ago.
There is a vast literature on limit theorems for general stationary random fields, and we refer to \citep{bolthausen82central,bradley07introduction,dedecker98central,dedecker01exponential} and the references therein. A main motivation of the recent developments was to extend the well-investigated dependence conditions for stationary sequences to random fields. 
However, the success so far has been mostly limited to Bernoulli random fields $\{X_j\}_{j\in\Zd}$: 
\equh\label{eq:Bernoulli}
X_j = f\circ T_j(\ccbb{\epsilon_k}_{k\in\Zd})
\eque
where $f:\R^\Zd\to\R$ is a measurable function, $\{T_j\}_{j\in\Zd}$ are the shift operators on $\R^\Zd$ such that for $w = \{w_k\}_{k\in\Zd}\in\R^\Zd$, $[T_j(w)]_k = w_{j+k}$ for all $k,j\in\Zd$, and $\{\epsilon_k\}_{k\in\Zd}$ are i.i.d.~random variables. We assume $\esp X_j = 0$ and $\esp X_j^2<\infty$. 

It is important to point out that in the case $d=1$, $\{\epsilon_j\}_{j\in\Z}$ is only assumed to be stationary. At the technical level, the reason for restricting $\{\epsilon_j\}_{j\in\Zd}$ to be i.i.d.~when $d\ge2$ is the following. Conditional expectations of random variables, in the form of $\esp (Y\mid\filF_j)$ with $\filF_j = \sigma(\epsilon_k:k\leq j, k\in\Zd)$, are involved in the proof.  Such conditional expectations are extensively used to express random variables of interest in terms of sums of orthogonal random variables. For example in one dimension, one can write
\equh\label{eq:telescoping}
Y = \sum_{j\in\Z}\bb{\esp(Y\mid\filF_j)-\esp(Y\mid\filF_{j-1})} =: \sum_jP_jY
\eque
for any regular random variable $Y$ with finite moment, and for all $j\neq j'$, $P_jY$ and $P_{j'}Y$ are uncorrelated. To extend this key decomposition to the random-field setup (Lemma~\ref{lem:L2} below) and maintain the orthogonality of the terms, one needs the following commuting property of the filtration:
\equh\label{eq:commuting}
\esp[\esp(Y\mid\filF_j)\mid\filF_k] = \esp[\esp(Y\mid\filF_k)\mid\filF_j] = \esp(Y\mid\filF_{j\wedge k})
\eque
with $(j\wedge k)_q = \min(j_q,k_q), q=1,\dots,d$. This identity, unfortunately, is not true for filtrations generated by general stationary $\{\epsilon_j\}_{j\in\Zd}$ except when they are i.i.d. Therefore, many arguments based on the telescoping~\eqref{eq:telescoping} fail to be generalized to high dimensions for arbitrary stationary $\{\epsilon_j\}_{j\in\Zd}$. 

A deeper reason for working under the assumption of $\{\epsilon_j\}_{j\in\Zd}$ being i.i.d.~is that the so-constructed random fields as in~\eqref{eq:Bernoulli} can be approximated by $m$-dependent random fields, but it is not clear how to construct an $m$-dependent approximation for general stationary random fields. For general stationary sequences in one dimension, instead of $m$-dependent approximation one can proceed alternatively by establishing martingale approximation, using the projections based on conditional expectations, and then apply the martingale central limit theorem. However, the martingale central limit theorem has been well known (e.g.~\citep{bolthausen82central}) to be much less practical in high dimensions than in one dimension, and limit theorems for stationary random fields established via martingale approximations have very stringent conditions \citep{basu79functional,nahapetian95billingsley,poghosyan98invariance,morkvenas84invariance}. So for general stationary random fields in form of~\eqref{eq:Bernoulli} with stationary $\{\epsilon_j\}_{j\in\Zd}$, it is still an open question that what should be the general approach to establish central limit theorems. 
This is the limitation of most of the results so far on weighted Bernoulli random fields. It might be true that the situation with stationary $\{\epsilon_j\}_{j\in\Zd}$ is actually much more complicated than in one dimension. An exceptional result for general stationary random fields is due to \citet{dedecker01exponential} via a very involved Lindeberg method, and the condition on weak dependence is much more complicated than the one-dimensional case. Another very recent result that goes beyond the i.i.d.~regime is in~\citep{volny15central}, although it is still not yet as  satisfactory as in one dimension. 

As discussed above, latest advances on limit theorems for stationary random fields have been mostly limited to Bernoulli random fields in form of~\eqref{eq:1} and~\eqref{eq:Bernoulli}. In particular, progress has been made on Wu's condition in the past few years. Observe that Wu's condition in one dimension~\eqref{eq:Wud=1} can be naturally extended to high dimensions, formally as
\equh\label{eq:Wu}
\sum_{j\in\Zd}\snn{X_j-X_j^*}_2<\infty,
\eque
with $X_j^* = f\circ T_j(\{\epsilon_k^*\}_{k\in\Zd})$ and $\{\epsilon_k^*\}_{k\in\Zd}$ similarly as in the case $d=1$. 
This extension is first considered by \citet{elmachkouri13central}, where invariance principles for unweighted partial sums are established. Limit theorems have been also established for fractional Brownian sheets \citep{wang14invariance} and set-indexed random fields \citep{bierme14invariance,elmachkouri13central}. All these results can be formulated as limit theorems for weighted Bernoulli random fields as in~\eqref{eq:1}, under Wu's condition and certain assumptions on $b_{n,j}$. Recent results on  Bernoulli random fields under other conditions include also~\citep{elmachkouri14orthomartingale}.

In this paper, we consider limit theorems for Bernoulli random fields under Hannan's condition (\eqref{eq:Hannan} below), and continue the development in \citep{volny14invariance}. It is proved in \citep{volny14invariance} that for stationary Bernoulli random field under Hannan's condition, 
\[
\ccbb{\frac1{n^{d/2}}\sum_{\vv 1\leq i\leq nt}X_i}_{t\in[0,1]^d}\weakto \sigma\{\B_t\}_{t\in[0,1]^d}
\]
as $n\to\infty$ in $D([0,1]^d)$, 
where the limiting random field is the Brownian sheet $\B$ up to a multiplicative constant $\sigma$, as in the case for $\{X_j\}_{j\in\Zd}$ being i.i.d. 
Here, we complement results in \citep{volny14invariance} by considering more general weights $b_{n,j}$: we extend a few of aforementioned results \citep{bierme14invariance, elmachkouri13central, wang14invariance} on limit theorems for weighted Bernoulli random fields under Wu's conditions to the strictly weaker Hannan's condition (except in one case where the results are not comparable; see Remark~\ref{rem:H}).

There are two key ingredients in the proofs here. One is the assumption on the i.i.d.~random variables discussed above: in particular, this assumption allows the approximation of the stationary random fields by $m$-dependent ones. The other is a moment inequality for weighted partial sums, in form of
\[
\nn{\sum_{j\in\Zd}b_{n,j}X_j}_p\leq C\pp{\sum_{j\in\Zd}b_{n,j}^2}^{1/2}
\]
for some $p\geq 2$. 
We establish such an inequality in Lemma~\ref{lem:1} under Hannan's condition. It plays the key role in bounding the error term in the $m$-dependent approximation of the random fields.  Such an inequality has been known under Wu's condition~\eqref{eq:Wu} \citep[Proposition 1]{elmachkouri13central}. (See also for a different extension of Wu's condition  proposed by \citep{truquet10moment}, where a similar moment inequality was established for the unweighted partial sums.) 
Our proof of the main result, Theorem~\ref{thm:1}, makes essential use of the two keys, and the proof is inspired  by \citet{bierme14invariance} (see also Remark~\ref{rem:BD}). Here we present a variation of the same idea, using $m$-dependent approximation instead of $m_n$-dependent approximation.

The paper is organized as follows. The main result, Theorem~\ref{thm:1}, is established in Section~\ref{sec:2}. As consequences, we present two applications. First, central limit theorems for partial sums over arbitrary index sets are investigated in Section~\ref{sec:3}. Second, invariance principles in \citep{bierme14invariance,elmachkouri13central,wang14invariance} are established under the Hannan's condition in Section~\ref{sec:4}.

\section{A central limit theorem}\label{sec:2}

Consider i.i.d.~random variables $\{\epsilon_i\}_{i\in\Zd}$ defined in a probability space $(\Omega,\calB,\proba)$. 
Set 
$\filF_i = \sigma(\epsilon_j:j\in\Zd, j\leq i), i\in\Zd$ and  $\filF_{i_q}\topp q = \sigma(\epsilon_j:j\in\Zd,j_q\leq i_q), q=1,\dots,d, i_q\in\Z$. Because of the independence of $\{\epsilon_j\}_{j\in\Zd}$, $\{\filF_j\}_{j\in\Zd}$ are commuting in the sense of~\eqref{eq:commuting}. Next, as in \citep{volny14invariance}, introduce the projection operator
\[
P_i = \prodd q1d P\topp q_{i_q} \qmwith P_{i_q}\topp q (\cdot) = \esp(\cdot\mid\filF_{i_q}\topp q) - \esp(\cdot\mid\filF_{i_q-1}\topp q).
\]
In this way, $P_{i_q}\topp q$ and $P_i$ are commuting operators from $L^2(\Omega,\calB,\proba)$ to $L^2(\Omega,\calB,\proba)$, due to the commuting property of the filtration. For more properties of these filtrations and operators, see \citep{khoshnevisan02multiparameter,wang14invariance}. The following decomposition based on these projection operators is useful.
\begin{Lem}\label{lem:L2}
Let $Y$ be a random variable measurable with respect to the $\sigma$-algebra $\filF_\infty = \sigma(\epsilon_j:j\in\Zd)$, with $\esp Y = 0, \esp |Y|^p<\infty$, for some $p\geq 2$. Then, 
\[
Y = \sum_{j\in\Zd}P_jY \defe \limm\sum_{j\in\{-m,\dots,m\}^d}P_jY \mbox{ in $L^p$.}
\]
\end{Lem}
\begin{proof}
By definition of $P_j$, 
\[
\sum_{j\in\{-m+1,\dots,m\}^d}P_jY = \esp(Y\mid\filF_{m\ind}) + \sum_{\delta\in\{-1,1\}^d\setminus\{\ind\}}(-1)^{\ip\delta}\esp(Y\mid\filF_{m\delta})
\]
with $\ip\delta = \summ q1d \indd{\delta_q=-1}$, $\ind = (1,\dots,1)\in\Zd$ and $m\delta ,m\ind\in\Zd$. By martingale convergence theorem, $\esp(Y\mid\filF_{m\ind})\to Y$ almost surely and in $L^p$. The other $2^d-1$ terms all converge to zero in $L^p$. Indeed, observe that for each $q=1,\dots,d$, 
\[
\limm\esp(Y\mid\filF_{-m}\topp q) = \esp\pp{Y\mmid\bigcap_{m\in\N}\filF_{-m}\topp q}
\]
almost surely and in $L^p$, by backwards martingale convergence theorem. By Kolmogorov's zero-one law, the limit is a constant and hence necessarily zero since $\esp Y = 0$. To complete the proof, it suffices to observe for each $\delta\in\{-1,1\}^d\setminus\{\vv1\}$ with $\delta_q=-1$ for some $q\in\{1,\dots,d\}$, $\snn{\esp(Y\mid\filF_{m\delta})}_p\leq \snn{\esp(Y\mid\filF\topp q_m)}_p$ for $j\in\Zd$.
\end{proof}
In view of Lemma~\ref{lem:L2}, throughout, an infinite sum of random variables is understood as the limit of partial sums over sequence of finite sets in the $L^p$ sense.

With projection operators defined above, the Hannan's condition states as
\equh\label{eq:Hannan}
\Delta_p(X) \defe \sum_{i\in\Zd}\nn{P_{\vv 0}X_i}_p<\infty,
\eque
for some $p\geq 2$.
We first give two lemmas on Bernoulli random fields under Hannan's condition. \begin{Lem}\label{lem:1}
Suppose $\Delta_p(X)<\infty$ for some $p\geq 2$. Then for all $\{a_i\}_{i\in\Zd}\in\ell^2(\Zd)$, 
\equh\label{eq:weighted}
\nn{\sum_{i\in\Zd}a_iX_i}_p\leq C_{p,d} \pp{\sum_{i\in\Zd}a_i^2}^{1/2}\Delta_p(X)
\eque
with $C_{p,d} = (p-1)^{d/2}$.
\end{Lem}

\begin{proof}
Observe that it suffices to show 
\[
\nn{\sum_{i\in\Lambda}a_iX_i}_p\leq C_{p,d}\pp{\sum_{i\in\Lambda}a_i^2}^{1/2}\Delta_p(X)
\]
for all finite $\Lambda\subset\Zd$. Then, by Lemma~\ref{lem:L2},
\[
\sum_{i\in\Lambda}a_iX_i = \sum_{i\in\Lambda}a_i\sum_{j\in\Zd}P_jX_i = \sum_{j\in\Zd}P_j\pp{\sum_{i\in\Lambda}a_iX_i}.
\]
Now let $Y$ be any random variable with $\esp |Y|^p<\infty$ in the same probability space, and recall that $P_j = \prodd q1dP_{j_q}\topp q$. Then,
\begin{align}
\nn{\sum_{j\in\{-m,\dots,m\}^d}P_jY}^2_p  &  = \nn{\sum_{j_1 = -m}^m P_{j_1}\pp{\sum_{j_2,\dots,j_d\in\{-m,\dots,m\}}\prodd q2d P_{j_q}Y}}_p^2\nonumber\\
& \leq (p-1)\summ {j_1}{-m}m\nn{P_{j_1}\pp{\sum_{j_2,\dots,j_d\in\{-m,\dots,m\}}\prodd q2d P_{j_q}Y}}_p^2\label{eq:MZ}\\
& = (p-1)\summ {j_1}{-m}m\nn{\sum_{j_2,\dots,j_d\in\{-m,\dots,m\}}\prodd q2d P_{j_q}(P_{j_1}Y)}_p^2\nonumber,
\end{align}
where in the first and last equalities above we used the commuting property of the projection operators, and the inequality is a Marcinkiewicz--Zygmund type inequality for one-dimensional martingales due to \citet[Theorem 2.1]{rio09moment}.
Iterating the same argument we arrive at, for all $m\in\N$, 
\equh\label{eq:MZd}
\nn{\sum_{j\in\{-m,\dots,m\}^d}P_jY}_p^2 \leq C_{p,d}^2\sum_{j\in\{-m,\dots,m\}^d}\nn{P_jY}_p^2.
\eque
Taking $Y = \sum_{i\in\Lambda}a_iX_i$, we obtain for all $m\in\N$,
\begin{align*}
\nn{\sum_{j\in\{-m,\dots,m\}^d}P_j\pp{\sum_{i\in\Lambda}a_iX_i}}_p & \leq C_{p,d}^2\sum_{j\in\{-m,\dots,m\}^d}\nn{P_j\pp{\sum_{i\in\Lambda}a_iX_i}}_p^2\\
& \leq C_{p,d}^2 \sum_{j\in\Zd}\pp{\sum_{i\in\Lambda}|a_i|\nn{P_{\vv0}X_{i-j}}_p}^2\\
& \leq C_{p,d}^2\sum_{j\in\Zd}\sum_{i\in\Lambda}a_i^2\nn{P_{\vv0}X_{i-j}}_p\sum_{\ell\in\Lambda}\nn{P_{\vv0}X_{\ell-j}}_p \\
& = C_{p,d}^2\Delta_p^2(X)\sum_{i\in\Lambda}a_i^2,
\end{align*}
where we applied triangle inequality and Cauchy--Schwarz inequality in the second and third steps, respectively.
Thus, we have shown~\eqref{eq:weighted}.
\end{proof}
\begin{Rem}\label{rem:Cp}
It is not clear to us whether the constant $C_{p,d} = (p-1)^{d/2}$ is optimal for $d\ge 2$. It is proved by \citet{rio09moment} that it is optimal when $d=1$. The constant $C_{p,d}$ will play a role when establishing tightness with entropy conditions for invariance principles. See Remark~\ref{rem:H} below.
\end{Rem}
\begin{Lem}\label{lem:2}
Suppose $\Delta_2(X)<\infty$. Then $\sum_{j\in\Zd}|\cov(X_{\vv0},X_j)|\leq \Delta_2^2(X)<\infty$.
\end{Lem}
\begin{proof}
Hannan's condition enables to write $X_i = \sum_jP_jX_i$. Since $\{P_j\}_{j\in\Zd}$ are orthogonal in the sense that $\esp[(P_jX)(P_kY)] = 0$ for all $j,k\in\Zd,j\neq k$ and $X,Y\in L^2(\Omega,\calB,\proba)$, it follows that
\begin{multline*}
\sum_{k\in\Zd}|\esp(X_{\vv0}X_k)| \leq \sum_{k\in\Zd}\sum_{i\in\Zd}\esp|(P_iX_{\vv0})(P_iX_k)|\\
\leq \sum_{k\in\Zd}\sum_{i\in\Zd}\nn{P_iX_{\vv0}}_2\nn{P_iX_k}_2 = \Delta_2^2(X).
\end{multline*}
\end{proof}
As a consequence, we introduce
\equh\label{eq:sigma}
\sigma^2 \defe\sum_{j\in\Zd}\cov(X_{\vv0},X_j)
\eque
which is finite under Hannan's condition. 

To state the main result, introduce $\vec b_n = \{b_{n,j}\}_{j\in\Zd}\in\ell^2(\Zd)$, $b_n := (\sum_{j\in\Zd}b_{n,j}^2)^{1/2}$. 
For $\{\vec b_n\}_{n\in\N}\subset\ell^2(\Zd)$, we are interested in
\[
S_n = \sum_{j\in\Zd}b_{n,j}X_j,
\] 
which by Lemma~\ref{lem:1} is defined in the $L^2$ sense under $\Delta_2(X)<\infty$, and moreover for  $\sigma_n^2 := \var(S_n)$,
\equh\label{eq:varSn1}
\sigma_n^2 \leq Cb_n^2\Delta_2^2(X)<\infty, n\in\N
\eque
for some constant $C>0$.
Our main result is the following.
\begin{Thm}\label{thm:1}
Let $\{X_i\}_{i\in\Zd}$ be a stationary Bernoulli random field as in~\eqref{eq:Bernoulli} satisfying Hannan's condition~\eqref{eq:Hannan}. 
If
\equh\label{eq:supbnj}
\limn\sup_{j\in\Zd}\frac{|b_{n,j}|}{b_n} = 0
\eque
and
\equh\label{eq:liminf}
\liminfn\frac{\sigma_n^2}{b_n^2}>0
\eque
hold, then
\equh\label{eq:CLT}
\frac{S_n}{\sigma_n}\weakto\calN(0,1).
\eque
\end{Thm}
The condition~\eqref{eq:liminf} is subtle as it involves both the coefficients and the dependence of underlying random fields (via $\sigma_n$). 
The following corollary is more convenient, as it imposes only conditions on coefficients. However, we see later in Example~\ref{example:1} that there are examples that satisfy the conditions in Theorem~\ref{thm:1}, but the conclusion of Corollary~\ref{coro:1} does not hold. 
Recall that for $k\in\Zd$, the shift operator yields $T_k\vec b_n = \{b_{n,j+k}\}_{j\in\Zd}$. Let $e_1,\dots,e_d$ be the $d$ canonical unit vector in $\Rd$.

\begin{Coro}\label{coro:1}
Let $\{X_i\}_{i\in\Zd}$ be a stationary Bernoulli random field as in~\eqref{eq:Bernoulli} satisfying Hannan's condition~\eqref{eq:Hannan}. Under the notations as in Theorem~\ref{thm:1},
if
\equh\label{eq:shift0}
\limn\frac{\snn{T_{e_q}\vec b_n - \vec b_n}_{\ell^2}}{b_n} = 0, \mfa q=1,\dots,d
\eque
 hold, then
\equh\label{eq:var}
\limn\frac{\sigma_n}{b_n} = \sigma
\eque
with $\sigma$ defined as in~\eqref{eq:sigma}, and  
\[
\frac{S_n}{b_n}\weakto\calN(0,\sigma^2).
\]
\end{Coro}
\begin{proof}[Proof of Corollary~\ref{coro:1}]
We first show~\eqref{eq:var}. 
Recall~\eqref{eq:varSn1}. 
Observe that 
\[
-2\sum_{k\in\Zd} b_{n,k}b_{n,k+j} = \snn{T_j\vec b_n-\vec b_n}_{\ell^2}^2 - \snn{\vec b_n}_{\ell^2}^2 - \snn{T_j\vec b_n}_{\ell^2}^2,
\]
and $\snn{T_j\vec b_n-\vec b_n}_{\ell^2} = o(b_n)$ for all fixed $j\in\Zd$, a consequence of~\eqref{eq:shift0}.  Therefore,
\equh\label{eq:varSn2}
\limn\frac1{b_n^2}\sum_{k\in\Zd}b_{n,k}b_{n,k+j} = 1 \mfa j\in\Zd.
\eque
 Thus, by the dominated convergence theorem,~\eqref{eq:varSn1} and~\eqref{eq:varSn2} imply~\eqref{eq:var}.

If $\sigma = 0$, then $\sigma_n^2/b_n^2\to 0$, and the central limit theorem is degenerate and trivially holds. If $\sigma>0$, then~\eqref{eq:liminf} holds. 
By Cauchy--Schwarz inequality, ~\eqref{eq:shift0} implies
\equh\label{eq:shift}
\limn\frac1{b_n^2}\sum_{j\in\Zd}\abs{b_{n,j+e_q}^2-b_{n,j}^2} = 0, \mfa q=1,\dots,d.
\eque
It has been shown in \citep[Lemma 8]{bierme15invariance}, using an idea from \citep{peligrad97central}, that~\eqref{eq:shift} implies~\eqref{eq:supbnj}. 
 The desired result now follows from Theorem~\ref{thm:1}.
\end{proof}
\begin{Rem}\label{rem:redundant}
Condition~\eqref{eq:shift0} was introduced in \citet[Theorem 3.1]{bierme14invariance}. Condition~\eqref{eq:supbnj} was also assumed there. It has been pointed out in \citep[Remark 3]{bierme15invariance} that~\eqref{eq:supbnj} was redundant. 
\end{Rem}

\begin{proof}[Proof of Theorem~\ref{thm:1}]
We proceed an $m$-dependent approximation argument. 
For each $m\in\N$, set $\calG_j\topp m = \sigma(\epsilon_i: i\in\Zd, |j-i|_\infty\leq m)$, $X_j\topp m = \esp(X_j\mid\calG_j\topp m), j\in\Zd$. 
In this way, $\{X\topp m_j\}_{j\in\Zd}$ is a ($2m+1$)-dependent stationary random field.
Write
\[
S_n\topp m = \sum_{j\in\Zd}b_{n,j}X_j\topp m \qmand \sigma_{m,n}^2 = \var(S_n\topp m).
\]
Observe that \begin{multline*}\label{eq:P0Xjm}
P_{\vv0}X_j\topp m= \sum_{\delta\in\{0,1\}^d}(-1)^{\delta_1+\cdots+\delta_d}\esp\bb{\esp(X_j\mid\calG_j\topp m)\mmid\filF_{-\delta}}\\
= \sum_{\delta\in\{0,1\}^d}(-1)^{\delta_1+\cdots+\delta_d}\esp\bb{\esp(X_j\mid\filF_{-\delta})\mmid\calG_j\topp m} = \esp\pp{P_{\vv0}X_j\mmid\calG_j\topp m},
\end{multline*}
where in the second equality we used the fact that the $\sigma$-algebras $\calG_j\topp m$ and $\filF_\ell$ are conditionally independent and hence commuting, because they are generated by independent random variables $\{\epsilon_j\}_{j\in\Zd}$. Thus,
\equh\label{eq:deltapm}
\Delta_p(X\topp m)\leq \Delta_p(X),
\eque
and $S_n\topp m$ is well defined in the $L^p$ sense if $\Delta_p(X)<\infty, p\geq 2$. 

We will approximate $S_n$ by $S_n\topp m$. To establish a central limit theorem for $m$-dependent random variables, we will apply a result due to \citet{heinrich88asymptotic}, which requires each partial sum to be of finite number of random variables. Therefore, we introduce a finite set $V_n\subset\Zd$ for each $n$ such that $|V_n|\to\infty$ and $\limn b_n^{-2}\sum_{j\in V_n}b_{n,j}^2 = 1$. Set
\[
S_{V_n}\topp m = \sum_{j\in V_n}b_{n,j}X_j\topp m \qmand \sigma_{m,V_n}^2 = \var(S_{V_n}\topp m).
\]
We first summarize a few estimates in the following lemma. 

\begin{Lem}\label{lem:bounds}
With the construction described above, 
\equh\label{eq:SnSnm}
\limm\limsup_{n\in\N}\frac{\var(S_n-S_n\topp m)}{\sigma_n^2} = 0, \limm\limsup_{n\in\N}\frac{|\sigma_{m,n}^2-\sigma_n^2|}{\sigma_n^2} = 0,
\eque
and with the choice of $V_n$ described above, for every $m$ large enough,
\equh\label{eq:SnSVn}
\limn\frac{\var(S_n\topp m-S_{V_n}\topp m)}{\sigma_{m,n}^2} = 0, \limn\frac{\sigma_{m,V_n}^2}{\sigma_{m,n}^2} = 1.
\eque
\end{Lem}
\begin{proof}[Proof of Lemma~\ref{lem:bounds}]
In the sequel, we let $C$ denote constant number independent from $n$ and $m$, but may change from line to line.
We first show the first part of~\eqref{eq:SnSnm}. 
Indeed, by Lemma~\ref{lem:1}, 
\[
\var(S_n - S_n\topp m) \leq Cb_n^2\pp{\sum_{j\in\Zd}\nn{P_{\vv0}(X_j-X_j\topp m)}_2}^2.
\]
Observe that for each $j$, $\snn{P_{\vv0}(X_j-X_j\topp m)}_2\leq \snn{X_j-X_j\topp m}_2\to 0$ as $m\to\infty$, and  that
\begin{multline*}
\sum_{j\in\Zd}\nn{P_{\vv0}(X_j-X_j\topp m)}_2\leq\sum_{j\in\Zd}\pp{\nn{P_{\vv0}(X_j)}_2+\snn{P_{\vv0}(X_j\topp m)}_2} \\
\leq  \Delta_2(X\topp m)+\Delta_2(X)\leq 2\Delta_2(X),
\end{multline*}
which is finite under Hannan's condition. By the dominated convergence theorem, $\limm\sup_{n\in\N}\var(S_n\topp m-S_n)/b_n^2 = 0$, and the first part of~\eqref{eq:SnSnm} follows from the assumption~\eqref{eq:liminf}. To see the second part, it suffices to observe
\[
\abs{\sigma_{m,n}^2-\sigma_n^2} \leq \var^{1/2}(S_n\topp m-S_n)\var^{1/2}(S_n\topp m+S_n).
\]
We have seen that $\sigma_n^2 \leq Cb_n^2$ in~\eqref{eq:varSn1}. Again by Lemma~\ref{lem:1} and~\eqref{eq:deltapm}, 
\[
\sigma_{m,n}^2\leq Cb_n^2\Delta_2^2(X\topp m)\leq Cb_n^2\Delta_2^2(X).
\]
Therefore, $\var(S_n\topp m+S_n)\leq 2(\sigma_{m,n}^2+\sigma_n^2)\leq Cb_n^2$, for all $m,n\in\N$. It then follows
\[
\limsupn\frac{\sabs{\sigma_{m,n}^2-\sigma_n^2}}{\sigma_n^2} \leq  C\limsupn \frac{b_n}{\sigma_n}\frac{\var^{1/2}(S_n\topp m-S_n)}{\sigma_n}.
\]
The second part of~\eqref{eq:SnSnm} now follows from the first part and~\eqref{eq:liminf}.

For~\eqref{eq:SnSVn}, to show the first part, using the same argument as above it suffices to observe 
\[
\frac{\var(S_n\topp m-S_{V_n}\topp m)}{\sigma_{m,n}^2}\leq C\frac{\sum_{j\notin V_n}b_{n,j}^2}{\sigma_{m,n}^2}\Delta^2_2(X\topp m)\leq C\frac{\sum_{j\notin V_n}b_{n,j}^2}{b_n^2}\frac{b_n^2}{\sigma_n^2}\frac{\sigma_n^2}{\sigma_{m,n}^2}\Delta_2^2(X),
\]
again by Lemma~\ref{lem:1} and~\eqref{eq:deltapm}. By the second part of~\eqref{eq:SnSnm}, for $m$ large enough, say $m\geq m_0$,  $\limsupn|\sigma_{m,n}^2-\sigma_n^2|/\sigma_n^2\leq 1/2$, whence
\equh\label{eq:m0}
\limsupn\frac{\sigma_n^2}{\sigma_{m,n}^2}\leq 2, m\geq m_0.
\eque
Therefore the first part of~\eqref{eq:SnSVn} follows, for $m\geq m_0$.
For the second part, observe that 
\[
|\sigma_{m,n}^2-\sigma_{m,V_n}^2| \leq \var^{1/2}(S_n\topp m-S_{V_n}\topp m)\var^{1/2}(S_n\topp m+S_{V_n}\topp m),
\] and by Lemma~\ref{lem:1} and~\eqref{eq:deltapm},
\[
\sigma_{m,V_n}^2 \leq C\pp{\sum_{j\in V_n}b_{n,j}^2}\Delta_2^2(X\topp m)\leq Cb_n^2\Delta_2^2(X).
\]
Thus, 
\equh\label{eq:2}
\frac{|\sigma_{m,n}^2-\sigma_{m,V_n}^2|}{\sigma_{m,n}^2}\leq C\pp{\frac{\var(S_n\topp m - S_{V_n}\topp m)}{\sigma_{m,n}^2}}^{1/2}\frac{b_n}{\sigma_n}\frac{\sigma_n}{\sigma_{m,n}}.
\eque
By~\eqref{eq:liminf},~\eqref{eq:m0} and  the first part of~\eqref{eq:SnSVn}, for $m\geq m_0$ the second part of~\eqref{eq:SnSVn} follows.
\end{proof}
Now we prove the desired central limit theorem~\eqref{eq:CLT} in three steps. \medskip

\noindent 1) We first show, for $m$ large enough,
\equh\label{eq:CLTmV}
\limn\frac{S_{V_n}\topp m}{\sigma_{m,V_n}}\weakto\calN(0,1).
\eque 
For this purpose, we apply the central limit theorem for $m$-dependent random variables due to \citet{heinrich88asymptotic}. 
We need also 
\equh\label{eq:mVn}
\limsupn \frac{b_n}{\sigma_{m,V_n}}<\infty,
\eque
which follows from~\eqref{eq:liminf} and~\eqref{eq:SnSVn}, for $m$ large enough. For~\eqref{eq:CLTmV},
the required conditions in Heinrich's theorem can be easily verified: for any $m\in\N$ large enough fixed,
\[
\frac1{\sigma_{m,V_n}^2}\sum_{j\in V_n}\esp\pp{b_{n,j}^2 X_j^{(m)2}}\leq \frac{b_n^2}{\sigma_{m,V_n}^2}\var(X_{\vv0}\topp m)\leq C< \infty
\]
for some constant $C$ and $n$ large enough, and for all $\epsilon>0$,
and
\begin{multline*}
\frac{m^{2d}}{\sigma_{m,V_n}^2}\sum_{j\in V_n}\esp\pp{b_{n,j}^2X_j^{(m)2}\indd{|X_j\topp m|\geq \epsilon m^{-2d}\frac{\sigma_{m,V_n}}{|b_{n,j}|}}}\\
\leq \frac{m^{2d}b_n^2}{\sigma_{m,V_n}^2}\esp\pp{X_{\vv0}^{(m)2}\indd{|X_{\vv0}\topp m|\geq \epsilon m^{-2d}/\sup_j\frac{|b_{n,j}|}{\sigma_{m,V_n}}}}\to 0 \mmas n\to\infty
\end{multline*}
where the last step is due to~\eqref{eq:mVn} and the assumption~\eqref{eq:supbnj}.\medskip

\noindent 2) Observe that
\[
\frac{S_n\topp m}{\sigma_{m,n}} = \frac{S_n\topp m-S_{V_n}\topp m}{\sigma_{m,n}} + \frac{S_{V_n}\topp m}{\sigma_{m,V_n}}\frac{\sigma_{m,V_n}}{\sigma_{m,n}}.
\] From~\eqref{eq:SnSVn} and~\eqref{eq:CLTmV}, it follows that for $m$ large enough,
\equh\label{eq:CLTm}
\frac{S_n\topp m}{\sigma_{m,n}}\weakto\calN(0,1). \medskip
\eque
3) At last, to show~\eqref{eq:CLT}, observe that
\[
\frac{S_n}{\sigma_n}-\frac{S_n\topp m}{\sigma_{m,n}} = \frac 1{\sigma_n}(S_n - S_n\topp m) + \frac{\sigma_{m,n}-\sigma_{n}}{\sigma_n\sigma_{m,n}}S_n\topp m.
\]
By Lemma~\ref{lem:bounds}, it follows that
\equh\label{eq:m-approximation}
\limm\limsupn\var\pp{\frac{S_n}{\sigma_n}-\frac{S_n\topp m}{\sigma_{m,n}}} = 0.
\eque
Therefore, applying \citep[Theorem 4.2]{billingsley68convergence} to~\eqref{eq:CLTm} and~\eqref{eq:m-approximation}, we have thus proved~\eqref{eq:CLT}. 
\end{proof}
\begin{Rem}\label{rem:BD}
The same $m_n$-dependent approximation as in \citep[Theorem 3.1]{bierme14invariance} can be applied here, once one notices that
\equh\label{eq:mn}
\limn\frac{\var(S_n-S_n\topp{m_n})}{b_n^2} = 0
\eque
holds (in the same way as in the proof of the first part of~\eqref{eq:SnSnm}) in place of \citep[Eq.~(3.4)]{bierme14invariance} for an appropriately chosen increasing sequence $\indn m$, and the rest of the proof therein can be carried out with minor changes. In order not to introduce too much duplication, we chose to present a different proof. Our result is more general also in the sense that we consider the normalization of $\sigma_n$ instead of $b_n$. 
\end{Rem}
\section{Central limit theorems for set-indexed partial sums}\label{sec:3}
In this section, we consider the case 
\[
S_n\equiv  S_{\Gamma_n} = \sum_{i\in\Gamma_n}X_i
\]
for a sequence of subsets $\indn\Gamma$ of $\Zd$ with the cardinality of subsets $|\Gamma_n| \to\infty$ as $n\to\infty$. This corresponds to the case $b_{n,j} = \indd{j\in\Gamma_n}$ and $b_n = |\Gamma_n|^{1/2}$. Then, in view of Corollary~\ref{coro:1}, it is easy to notice that~\eqref{eq:shift0} is equivalent to
\equh\label{eq:partial}
\limn\frac{|\partial\Gamma_n|}{|\Gamma_n|} = 0,
\eque
where $\partial\Gamma_n = \{i\in\Gamma_n: \exists j\notin\Gamma_n, |i-j|_\infty=1\}$ is the boundary set of $\Gamma_n$. Indeed, if we identify $\Gamma_n$ with an element in $\ell^2(\Zd)$ via $b_{n,j} = \indd{j\in\Gamma_n}$, then for each $q=1,\dots,d$ we have $\snn{T_{e_q}\Gamma_n - \Gamma_n}_{\ell^2}^2 \leq 2|\partial \Gamma_n| \leq \summ m1d\snn{T_{e_m}\Gamma_n-\Gamma_n}_{\ell^2}^2$. We have thus obtained the following.
\begin{Coro}
For a Bernoulli random field with $\Delta_2(X)<\infty$, and a sequence of subsets $\indn\Gamma$ of $\Zd$ satisfying $|\Gamma_n|\to\infty$ and~\eqref{eq:partial},
\equh\label{eq:conv0}
\frac{S_{\Gamma_n}}{|\Gamma_n|^{1/2}}\weakto\calN(0,\sigma^2)
\eque
with $\sigma^2$ given in~\eqref{eq:sigma}.
\end{Coro}

In the rest of this section, we discuss what happens if we are interested in the convergence of
\equh\label{eq:conv}
\frac{S_n}{\sigma_n}\weakto\calN(0,1).
\eque
This follows from~\eqref{eq:liminf}, by Theorem~\ref{thm:1}. 
 To see the role of the condition~\eqref{eq:liminf}, we provide two examples. First, by Example~\ref{example:1}, we show that condition~\eqref{eq:liminf} cannot be removed: otherwise~\eqref{eq:conv} may no longer hold under Hannan's condition. Second, by Example~\ref{example:2}, we show that the assumption in Corollary~\ref{coro:1} is strictly stronger than~\eqref{eq:liminf}, in the sense that there are examples satisfying~\eqref{eq:liminf}, but the conclusion of Corollary~\ref{coro:1} does not hold. Note also that Example~\ref{example:1} also shows that when $S_n/b_n\weakto\calN(0,\sigma^2)$ with $\sigma^2 = 0$, one should not expect $S_n/\sigma_n$ to converge, without further assumptions.
 
 For the sake of simplicity, both examples are given in one dimension. Let $\{\epsilon_i\}_{i\in\Z}$ be the i.i.d.~random variables that generate the Bernoulli random field~\eqref{eq:Bernoulli}.
\begin{Example}\label{example:1}
Consider $\Gamma_n = \{0,1,\dots,n-1\}$. We construct an example such that $S_n/\sigma_n$ converges to different limits along different subsequences. 

Suppose that there exists a collection of mutually independent random variables $\{\zeta_n\topp k\}_{n\in\Z,k\in\N}$ such that for each $k\in\N$, $\{\zeta\topp k_n\}_{n\in\Z}$ are i.i.d., and for each $n$, $\zeta\topp k_n$ is $\sigma(\epsilon_n)$-measurable. We further assume that $\esp\zeta_n\topp k = 0, \var(\zeta_n\topp k) = 1$. A detailed construction is given at the end.

For coefficients $\{\alpha_k\}_{k\in\N}$ satisfying $\sum_k|\alpha_k|<\infty$ and a sequence of increasing positive integers $\{n_k\}_{k\in\N}$, set
\[
W_n\topp k = \alpha_k(\zeta\topp k_n - \zeta_{n-n_k}\topp k), k\in\N \qmand
X_n = \sif k1 W\topp k_n.
\]
Observe that $P_{0}X_n = \sif k1P_{0}W_n\topp k$, which equals $-\alpha_\ell\xi_{0}\topp \ell$ if $n = n_\ell$ for some $\ell\in\N$, and $0$ otherwise. Thus, $\Delta_2(X) = \sif \ell1 |\alpha_\ell|<\infty$.

Write $S_n = S_{\Gamma_n} = \summ i0{n-1}X_i$ and $S_n(W\topp k) = \summ i0{n-1}W\topp k_i$. So
\[
S_n = \sif k1 S_n(W\topp k).
\] By independence, 
\[
\esp(S_{n_k}(W\topp\ell))^2 = \left\{
\begin{array}{ll}
2n_\ell\alpha_\ell^2 & \ell\leq k\\
2n_k\alpha_\ell^2 & \ell>k
\end{array}
\right.,
\]
and 
\[
\var(S_{n_k}) = \sif\ell1\var(S_{n_k}(W\topp\ell)) = \summ \ell1{k-1}2n_\ell\alpha_\ell^2 + \sif\ell{k+1}2n_k\alpha_\ell^2 + 2n_k\alpha_k^2.
\]
One can choose $\alpha_k$ and $n_k$ so that 
\equh\label{eq:SW}
\var(S_{n_k})\sim \var(S_{n_k}(W\topp k)) = 2n_k\alpha_k^2 \mmas k\to\infty.
\eque
For example, taking $\alpha_k = 2^{-k^2}$ and $n_k = 2^{3k^2}$ $k\in\N$, it  yields $\var(S_{n_k})\sim \var(S_{n_k}(W\topp k)) = 2^{k^2+1}$. 

Now in view of~\eqref{eq:SW}, for our purpose it suffices to choose $\zeta_k$ appropriately such that
\equh\label{eq:Z}
Z_k \defe\frac{S_{n_k}(W\topp k)}{\alpha_k\sqrt{n_k}}
\eque
converge to different limits along even and odd sequences. 

To do so, we now give an explicit construction of $\{\zeta_n\topp k\}_{n\in\Z,k\in\N}$. For the sake of simplicity, consider $(\Omega,\calB,\proba) = ([0,1]^\Z,\calB([0,1])^\Z,{\rm Leb}^\Z)$, and $\epsilon_n(\omega) = \omega_n, \omega\in\Omega$. Here, $T$ is the shift operator although we do not use it explicitly. In this way, for any sequence $\{d_k\}_{k\in\N}$ with $d_k\in[0,1]$, we choose a family of sets $\{A^\pm_k\}_{k\in\N}\subset\calB([0,1])$ such that $A^+_k\cap A^-_k = \emptyset$, $\mu(A^\pm_k) = d_k/2$,
 and set 
\[
\zeta_n\topp k(\omega) = \frac1{\sqrt{d_k}}(\ind_{A_k^+}-\ind_{A_k^-})(\omega_n), n\in\Z,k\in\N.
\]
So for $n\neq n'$, $\zeta_n\topp k$ and $\zeta_{n'}\topp{k'}$ are independent. 
In order that $\{\zeta_n\topp k\}_{n\in\Z,k\in\N}$ satisfy the conditions that we assumed at the beginning, it remains to choose $\{A_k^+,A_k^-\}_{k\in\N}$ such that for each fixed $n$, $\{\zeta_n\topp k \}_{k\in\N}$ are mutually independent. This can be done via a variation of dyadic expansion as follows. First, pick $A_1^+ := (0,d_1/2], A_1^- := (d_1/2,d_1]$. Suppose $A_k^+,A_k^-$ have been selected for $k\in\N$. Then each of $A_k^+,A_k^-$ and $(0,1]\setminus (A_k^+\cup A_k^-)$ can be expressed as a disjoint union of left-open-right-closed intervals, and together these intervals form  a partition of $(0,1]$, say $(0,1]  = \bigcup_{j=1}^{j_k}(a_{j,k},b_{j,k}]$. Now, set 
\begin{align*}
A_{k+1}^+ &:= \bigcup_{j=1}^{j_k} \Bigg(a_{j,k}, a_{j,k}+ (b_{j,k}-a_{j,k})\frac{d_{k+1}}2\Bigg]\\
A_{k+1}^- &:= \bigcup_{j=1}^{j_k} \Bigg(a_{j,k}+ (b_{j,k}-a_{j,k})\frac{d_{k+1}}2,a_{j,k}+ (b_{j,k}-a_{j,k}){d_{k+1}}\Bigg].
\end{align*}
The so-constructed $\{\zeta_n\topp k\}_{n\in\Z,k\in\N}$ are then mutually independent.

Now set $d_{k}=1$ for $k$ even and $d_{k} = 1/n_{k}$ for $k$ odd. For $\{A^\pm_k\}_{k\in\N}$ and $\{\zeta\topp k_n\}_{n,k\in\N}$ described above, when $k$ is even, $Z_k$ in~\eqref{eq:Z} becomes
\[
\frac1{\sqrt{n_k}}\summ i0{n_k-1}(\zeta_i\topp k-\zeta_{i-n_k}\topp k)
\]
which is the normalized sum of $2n_k$ Rademacher random variables, and thus $Z_{2k}\weakto\calN(0,2)$ as $k\to\infty$. At the same time, for $k$ odd, $Z_k$ in~\eqref{eq:Z} becomes 
\[
\sum_{i=0}^{n_k-1}(\zeta_i\topp k-\zeta_{i-n_k}\topp k),
\]
which is the sum of $2n_k$ i.i.d.~random variables with $\proba(\zeta_1\topp k = \pm1) = 1/(2n_k)$ and $\proba(\zeta_1\topp k = 0) = 1-1/n_k$. Clearly as $k\to\infty$ $Z_{2k-1}$ has a non-degenerate limiting distribution which is not Gaussian. So $S_n/\sigma_n$ does not converge.

\end{Example}
\begin{Example}\label{example:2}
Consider $X_i = \epsilon_i-\epsilon_{i-1}$. Observe that $X_i$ and $X_j$ are uncorrelated if $|i-j|\geq 2$. Therefore, this stochastic process satisfies $\Delta_2(X)<\infty$. We now construct a sequence of subsets $\indn\Gamma$ such that $\liminfn\sigma_n/b_n>0$ but $\limn\sigma_n/b_n$ does not exist. 

We construct $\Gamma_n$ iteratively. Set $\Gamma_1 = \{0,1\}$. For $n\in\N$, set $\Gamma_{n+1} = \Gamma_n\cup B_n$ with
\[
B_n = \left\{
\begin{array}{ll}
\{a_n+2,a_n+3,\dots,a_n+2^n+1\} & n \mbox{ even}\\
\{a_n+2,a_n+4,\dots,a_n+2\cdot 2^n\} & n \mbox{ odd}
\end{array}
\right.
\]
with $a_n = \max\{j:j\in\Gamma_n\}$. By construction, $\var(S_{\Gamma_{n+1}}) = \var(S_{\Gamma_n}) + \var(S_{B_n})$, and $\var(B_n) = 2\esp\epsilon_0^2$ for $n$ even, and $2^{n+1}\esp \epsilon_0^2$ for $n$ odd. At the same time, $|\Gamma_n| = 2^n$. It is clear that the desired result follows. 
\end{Example}

\section{Invariance principles for Gaussian random fields}\label{sec:4}
In this section, we present two invariance principles for weighted Bernoulli random fields. Let $\T$ be an index set equipped with a pseudo-metric. Consider random fields in form of
\equh\label{eq:Snt}
S_n(t) = \sum_{j\in\Zd}b_{n,j}(t)X_j, t\in\T.
\eque
Under Hannan's condition on $\{X_j\}_{j\in\Zd}$ and appropriate assumptions on the  coefficients $b_{n,j}(t)$, we shall establish, for an increasing sequence of positive numbers $\indn b$, 
\equh\label{eq:G}
\ccbb{\frac{S_n(t)}{b_n}}_{t\in \T}\weakto\ccbb{\G_t}_{t\in \T}
\eque
where $\G$ is a zero-mean Gaussian process. The space of weak convergence will be specified below.
Most results improve earlier ones \citep{bierme14invariance,elmachkouri13central,wang14invariance}, in the sense that Wu's condition is replaced by Hannan's condition.

We first provide an overview on how to establish~\eqref{eq:G}, illustrating how previous  proofs can be adapted without much changes. To establish such an invariance principle, we proceed as in the standard two-step proof: we first show convergence of finite-dimensional distributions and then tightness. To show the convergence of finite-dimensional distributions, we first remark that marginally, for
$b_{n}(t)\defe\spp{\sum_{j\in\Zd}b_{n,j}^2(t)}^{1/2}$,
one should expect 
\equh\label{eq:CLTt}
\frac{S_n(t)}{b_n(t)}\weakto\calN(0,\sigma^2), \mfa t\in \T
\eque
with $\sigma^2$ as in~\eqref{eq:sigma} as a consequence of Theorem~\ref{thm:1}.
Comparing this with~\eqref{eq:G}, it suggests that $\limn b_n^2(t)/b_n^2 = \var(\G_t)/\sigma^2$.
 Moreover, by Cramer--Wold's device, for the weak convergence to hold, we need to show, for all $\lambda\in\R^m, t\in \T^m, m\in\N$,
\equh\label{eq:fdd}
\frac1{b_n}\summ r1m\lambda_rS_n(t_r)\weakto\calN(0,\Sigma_{\lambda,t}^2) \qmwith \Sigma_{\lambda,t}^2 = \var\pp{\summ r1m\lambda_r\G_{t_r}}.
\eque
The linear combinations of finite-dimensional distributions can again be represented as a linear random field via
\[
\summ r1m\lambda_rS_n(t_r) = \sum_{j\in\Zd} \wt b_{n,j}X_j \qmwith \wt b_{n,j} = \summ r1m\lambda_rb_{n,j}(t_r),
\]
to which one can apply Theorem~\ref{thm:1} again. This is the standard procedure to establish finite-dimensional convergence of linear random fields. In our setup we have thus proved the following as a consequence of Theorem~\ref{thm:1}. Write $\wt b_n = (\sum_j \wt b_{n,j}^2)^{1/2}$.
\begin{Prop}\label{prop:fdd}
Consider random fields in form of~\eqref{eq:Snt} with $\{X_j\}_{j\in\Zd}$ satisfying Hannan's condition $\Delta_2(X)<\infty$. Suppose there exists a sequence of real numbers $\indn b$ such that 
\begin{itemize}
\item[(i)] for all $\lambda\in\R^m, t\in \T^m, m\in\N$, $\{\wt b_{n,j}\}_{j\in\Zd, n\in\N}$ satisfy the assumptions in Theroem~\ref{thm:1} and that $\wt b_n/b_n$ converges to a constant as $n\to\infty$, and 
\item[(ii)] for a zero-mean Gaussian process $\G$,
\equh\label{eq:covariance}
\limn\frac1{b_n^2}\esp(S_n(t)S_n(\tau)) = \esp(\G_t\G_\tau), \mfa t,\tau\in \T.
\eque
\end{itemize}
Then, the convergence of finite-dimensional distributions~\eqref{eq:fdd} holds.
\end{Prop}
We highlight that to apply Proposition~\ref{prop:fdd}, the essential work consists of verifying the assumptions on $\wt b_{n,j}$, and computing the covariance~\eqref{eq:covariance}. Both of these two steps are independent from the choice of dependence assumption on $\{X_j\}_{j\in\Zd}$. For invariance principles to be established below, these computations have been carried out in earlier proofs (under stronger assumptions on $\{X_j\}_{j\in\Zd}$) and can be borrowed here without any changes.

For the tightness, the moment inequality~\eqref{eq:weighted} in Lemma~\ref{lem:1} plays an important role. Similar inequalities have been used to establish tightness in the aforementioned work, and the proofs can be adapted with little extra effort in most the cases. See, however, Remark~\ref{rem:H} for an exception.

Below we present two improvements of earlier results. We only sketch the proofs in order not to introduce too much duplications.

\subsection{Invariance principles for self-similar set-indexed Gaussian fields}
Let $\mu$ be a $\sigma$-finite measure on $\Rd$. Consider 
\[
S_n(A) := \sum_{j\in\Zd}b_{n,j}(A)X_j \qmwith b_{n,j}(A) := \mu(nA\cap R_j)^{1/2}, A\in\calA
\]
where $R_j$ is the set of unit cube in $\Rd$ with lower corner $j\in\Zd$, and  $\calA$ is a class of Borel sets of $\Rd$, equipped with pseudo-metric $\rho(A,B) = \mu(A\triangle B)^{1/2}$. For $\mu$ being the Lebesgue measure, this framework has been considered for example in \citep{alexander86uniform,dedecker01exponential,elmachkouri13central}. The generalization to other measures, even for the i.i.d.~$\{X_j\}_{j\in\Zd}$, was first proposed by \citet{bierme14invariance}. In particular, they assume the measure $\mu$ to satisfy the following.
\begin{Assumption}\label{assumption:1}
$\mu$ is a $\sigma$-finite measure on $(\Rd,\calB(\Rd))$, absolutely continuous with respect to the Lebesgue measure, and such that 
\begin{itemize}
\item[(i)] There exists $\beta>0$ such that $\mu(nA) = n^\beta\mu(A)$ for all $n\in\N, A\in\calB(\Rd)$.

\item[(ii)] $\limsup_{\pi(j)\to\infty} \mu(R_j)<\infty$ and 
\[
\lim_{\pi(j)\to\infty}\frac{|\mu(R_{j+e_q}) - \mu(R_j)|}{\mu(R_j)} = 0, q=1,\dots,d
\]
with $\pi(j) = \min_{q=1,\dots,d}|j_q|, j\in\Zd$.
\end{itemize}
\end{Assumption}
Furthermore, they also worked with regular Borel sets $A$, that is, for the boundary set $\partial A$ of $A\subset\Rd$, ${\rm Leb}(\partial A) = 0$. The size and complexity of such classes are normally described via covering numbers: the smallest number of $\rho$-balls with radius $\epsilon$ to cover $\calA$, denoted by  $N(\calA,\rho,\epsilon)$. The entropy numbers are given by $H(\calA,\rho,\epsilon) = \log N(\calA,\rho,\epsilon)$. 

The following result is the counterpart of \citep[Theorem 4.5]{bierme14invariance},  replacing Wu's condition by Hannan's condition. See Remark~\ref{rem:H} below on the comparison with conditions under Wu's condition. For concrete examples on self-similar set-indexed random fields as applications, see \citep{bierme14invariance}. 

\begin{Thm}\label{thm:WIP}
Let $\mu$ be a measure on $\Rd$ satisfying Assumption~\ref{assumption:1}, and let $\calA$ be a class of regular Borel sets of $\Rd$ such that $\mu(A)<\infty$ for all $A\in\calA$. Assume further that one of the following conditions holds.
\begin{itemize}
\item [(i)] There exists $p\geq 2$ such that 
\equh\label{eq:entropy}
\int_0^1N(\calA,\rho,\epsilon)^{1/p}d\epsilon<\infty \qmand \Delta_p<\infty.
\eque
\item[(ii)] 
 There exists $\gamma\in(0,2/d]$, such that
\equh\label{eq:H}
\int_0^1H(\calA,\rho,\epsilon)^{1/\gamma}d\epsilon<\infty \qmand \sup_{p> 2}\frac{\Delta_p(X)}{p^{1/\gamma-d/2}}<\infty.
\eque
\end{itemize}
Then, 
\[
\ccbb{\frac{S_n(A)}{n^{\beta/2}}}_{A\in\calA} \weakto \sigma\{\G(A)\}_{A\in\calA}
\]
in the space of continuous functions on $\calA$ equipped with supremum norm,
 where $\sigma$ is as in~\eqref{eq:sigma} and $\G$ is a zero-mean Gaussian process with covariance $\cov(\G(A),\G(B)) = \mu(A\cap B)$. 
\end{Thm}
\begin{proof}
First, by \citep[Proposition 4.2]{bierme14invariance}, for each $A\in\calA$, $\{b_{n,j}(A)\}_{n,j}$ satisfy the assumptions of Theorem~\ref{thm:1}, and $b_n^2(A) = n^\beta\mu(A)$. This tells the order of normalization should be $n^{\beta/2}$. 

To show the convergence of finite-dimensional distributions, we apply Proposition~\ref{prop:fdd}. The verifications of conditions and the computations of covariance, all based on definitions of $b_{n,j}(A)$ and properties of $\mu$ only, have been carried out in the proof of \citep[Theorem 4.3]{bierme14invariance}.

Now we show the tightness. First consider assumption (i).  As in \citep{bierme14invariance}, we apply \citep[Theorem 11.6]{ledoux91probability}, which states, if for some constant $C>0$, $p\geq 2$,
\equh\label{eq:LT}
\frac1{n^{\beta/2}}\nn{S_n(A)-S_n(B)}_p\leq C \rho(A,B) \mfa n\in\N, A,B\in\calA,
\eque
and the first part of~\eqref{eq:entropy} holds, then
\[
\lim_{\eta\downarrow 0}\sup_{n\in\N}\esp\pp{\sup_{\substack{A,B\in\calA\\\rho(A,B)<\eta}}\frac{|S_n(A)-S_n(B)|}{n^{\beta/2}}} = 0,
\]
which yields the tightness. It remains to remark that~\eqref{eq:LT} follows from~\eqref{eq:weighted}.
 
 Now consider assumption (ii). 
Consider the Young function $\psi_\gamma(y) = \exp({y^\gamma})-1$ and the Orlicz norm $\nn X_{\psi_\gamma} := \inf\{a>0:\esp\psi_\gamma(|X|/a)\leq1\}$. It is well known (e.g.~\citep[Lemma 4]{elmachkouri13central}) that 
\[
\nn X_{\psi_\gamma}\leq C\sup_{p>2}\frac{\nn X_p}{p^{1/\gamma}}. 
\]
Now Lemma~\ref{lem:1} implies for all $n\in\N, A,B\in\calA$,
\begin{multline}\label{eq:4}
\frac1{n^{\beta/2}}\nn{S_n(A)-S_n(B)}_{\psi_\gamma}\\
\leq 
\frac C{n^{\beta/2}}\sup_{p>2}\frac{\nn{S_n(A)-S_n(B)}_p}{p^{1/\gamma}}\leq
C \rho(A,B)\sup_{p>2}\frac{C_{p,d}}{p^{1/\gamma}}\Delta_p(X).
\end{multline}
Recall that $C_{p,d} = (p-1)^{d/2}$. Again by \citep[Theorem 11.6]{ledoux91probability}  the tightness now follows from assumption~\eqref{eq:H}.
\end{proof}

\begin{Rem}\label{rem:H}
Under assumption (i), compared to \citep{bierme14invariance} we simply  replace Wu's condition by Hannan's condition, and thus strictly improve the results. However, the results obtained here in part (ii) are not comparable with those under Wu's condition: when working with Wu's condition, for the second part of~\eqref{eq:H} one can actually assume the strictly weaker assumption 
\equh\label{eq:gammaWu}
\sup_{p>2}\frac{\sum_{j\in\Zd}\snn{X_j-X_j^*}_p}{p^{1/\gamma-1/2}}<\infty.
\eque
To establish this condition, the only difference from the above argument is in the second inequality of~\eqref{eq:4}, where a similar moment inequality as in~\eqref{eq:weighted} is used, except that the constant $C_{p,d}$ is taken as $C_{p,d} = \sqrt{2p}$ for all $d\ge 1$ \citep[Proposition 1]{elmachkouri13central}. See for example \citep{elmachkouri13central,bierme14invariance}. (This constant can actually be replaced by the smaller one $C_{p,d}=\sqrt{p-1}$: it suffices to follow the same proof and replace the constant $\sqrt{2p}$ in \citep[Eq.(10)]{elmachkouri13central} by $\sqrt{p-1}$, due to \citep[Theorem 2.1]{rio09moment}.)

In other words, when replacing Wu's condition by Hannan's condition on the weak dependence of the stationary Bernoulli random fields, the condition on entropy numbers \eqref{eq:H} is strictly strengthened here. This is due to different constants in the moment inequalities of weighted partial sums as in \eqref{eq:weighted} in Lemma~\ref{lem:1} under different conditions: this constant $C_{p,d}$ plays a key role in the second inequality in~\eqref{eq:4}. For our approach here, the constant is essentially due to a Marcinkiewicz--Zygmund type inequality applied  iteratively $d$ times to orthomartingale differences in~\eqref{eq:MZd}; so the power $d$ reflects the dimension. For the moment inequality under Wu's condition in \citep{elmachkouri13central}, a dimension-free argument is applied: essentially a one-dimensional martingale is embedded into the random field, and thus the constant inherits the one from one-dimensional martingale inequalities.

It is not clear to us whether the constant $C_{p,d}$ in Lemma~\ref{lem:1} can be chosen to be independent from $d$. Such a choice would weaken the assumption (ii) in Theorem~\ref{thm:WIP}.

\end{Rem}
\subsection{An invariance principle for fractional Brownian sheet}
Consider a linear random field $\{Y_j\}_{j\in\Zd}$ in form of
\[
Y_j = \sum_{k\in\Zd}a_kX_{j-k}, j\in\Zd,
\]
with $\sum_ka_k^2<\infty$.
Invariance principles for 
\[
S_n(t) = \sum_{\vv 1\leq j\leq nt}Y_j,  t \in[0,1]^d
\]
with $nt = (nt_1,\dots,nt_d)$ have been studied in the literature. Observe that
\[
S_n(t) = \sum_{j\in\Zd}b_{n,j}(t) X_j \qmwith b_{n,j}(t) = \sum_{\vv1\leq i\leq nt}a_{i-j}.
\]
The following theorem generalizes \citep[Theorem 3]{wang14invariance}. In particular, \citep{wang14invariance} considered the case that $\{a_j\}_{j\in\Zd}$ is of the product form: there exist real numbers $\{a_{j_q}\topp q\}_{j_q\in\Z}, q=1,\dots,d$ such that 
\[
a_j = \prodd q1d a_{j_q}\topp q.
\]
Introduce also  $b_{n,j}\topp q = \summ i1n a_{i-j_q}\topp q$ and $b_n(q) = (\sum_{j\in\Z}b_{n,j}^{(q)2})^{1/2}$. Examples on coefficients satisfying the assumption below can be found in \citep[Example 2]{wang14invariance}. 
\begin{Thm}
Suppose there exists $H\in(0,1)^d$ such that 
\[
\limn \frac{b_{\floor{ns}}^2(q)}{b_n^2(q)} = s^{2H_q}, \mfa s\in[0,1], q=1,\dots,d,
\]
and there exists $p$ such that
\[
p\geq 2, p>\max_{q=1,\dots,d}\frac1{H_q} \mand \Delta_p(X)<\infty.
\]
Then, $\{S_n(t)/b_n\}_{t\in[0,1]^d}$ converges weakly in $D([0,1]^d)$ to a fractional Brownian sheet $\G^H$ with Hurst index $H$, a zero-mean Gaussian process with covariance
\[
\cov(\G^H_s,\G^H_t) = \frac1{2^d}\prodd q1d\pp{s_q^{2H_q}+t_q^{2H_q}-|t_q-s_q|^{2H_q}}, s,t\in[0,1]^d.
\]
\end{Thm}
\begin{proof}
To show the convergence of finite-dimensional distributions, the conditions in Proposition~\ref{prop:fdd} have been verified as in \citep[proof of Proposition 1]{wang14invariance}; actually, there a different set of conditions in \citep[Definition 1]{wang14invariance} on $b_{n,j}$  were verified. The equivalence between conditions there and ours were pointed out by \citet[Remark after Theorem 3.1]{bierme14invariance} (see also~Remark~\ref{rem:redundant}). 
We also point out that the conditions in~\citep{wang14invariance} were actually redundant: in \citep[Definition 1]{wang14invariance}, Eq.(8) implies Eq.(9) by Cauchy--Schwarz inequality. 

To show the tightness, by \citep[Corollary 3]{lavancier05processus_TR}, it suffices to show, for some $\beta>1, p>0$, 
\equh\label{eq:L}
\nn{S_n(t)}_p\leq Cb_n\prodd q1d t_q^{\beta/p}, t\in[0,1]^d.
\eque
For this purpose, by~\eqref{eq:weighted},
\equh\label{eq:23}
\nn{S_n(t)}_p \leq C \pp{\sum_{j\in\Zd}b_{n,j}^2(t)}^{1/2}\Delta_p(X),
\eque
which could lead to the desired condition~\eqref{eq:L}. This plan can be carried out as in \citep[Proposition 2]{wang14invariance}, with Eq.~(23) therein replaced by~\eqref{eq:23} above and no other changes. We omit the details.
\end{proof}

\subsection*{Acknowledgement} 
We would like to thank J\'er\^ome Dedecker for pointing the reference \citep{rio09moment} to us, and the AE and two anonymous reviewers for several suggestions that helped us improve the paper. Klicnarov\'a's research was partially supported by Czech Science Foundation (project number P201/11/P164). Wang's research was  partially supported by NSA grant H98230-14-1-0318.


\bibliographystyle{apalike}
\bibliography{../include/references}

\end{document}